\documentclass{amsart}
\usepackage{color,hyperref,mathtools,tikz,slashbox}
\usepackage{amssymb,latexsym,fullpage}
\usepackage{graphicx}
\usepackage{url}		

\theoremstyle{plain}
\numberwithin{equation}{section}
\newtheorem{thm}{Theorem}[section]

\newtheorem{lem}[thm]{Lemma}

\newtheorem{defn}[thm]{Definition}
\newtheorem{prop}[thm]{Proposition}
\newtheorem{cor}[thm]{Corollary}
\newtheorem{conj}[thm]{Conjecture}
\newtheorem{rmk}[thm]{Remark}
  
\DeclareMathOperator{\negF}{negF}
\DeclareMathOperator{\Pal}{Pal}

\newcommand{\RR}{\mathbb{R}}

\newcommand{\symm}{\mathfrak{S}}
\newcommand{\seqnum}[1]{\href{http://oeis.org/#1}{\underline{#1}}}

\begin{document}

\title{Further Combinatorics and Applications of Two-Toned Tilings}
\author{Robert Davis}
\address{Department of Mathematics\\
         Colgate University\\
         13 Oak Dr.\\
         Hamilton, NY 13346}
\email{rdavis@colgate.edu}
\author{Greg Simay}
\email{gregsimay@yahoo.com}

\begin{abstract}
	Integer compositions, integer partitions, Fibonacci numbers, and generalizations of these have recently been shown to be interconnected via two-toned tilings of horizontal grids.
	In this article, we present refinements of two-toned tilings, describe functions which analyze them, and apply these to generalizations of integer compositions and partitions which interpolate between the two.
\end{abstract}

\maketitle

\tableofcontents 

\section{Introduction}

An \emph{$n$-tiling} is an arrangement of \emph{tiles} of sizes $1\times 1$ through $1 \times n$ contained inside of a $1 \times n$ grid, covering the grid such that tiles may only intersect along common boundaries.
We say the \emph{length} of a $1 \times k$ tile is $k$.  
When $n$ is understood or unimportant, we may simply call an $n$-tiling a \emph{tiling}.
The tiles may come in various colors; in this article, we consider white and red tiles satisfying certain conditions.
The combinatorics of such tilings were initially explored in \cite{twotoned}, where they were used to determine the number of compositions of an integer with at least or exactly $p$ parts $k$, as well
as general formulas for positively-indexed, generalized Fibonacci numbers.

\begin{defn}
	For nonnegative integers $r$ and $n$, denote by $a(r,n)$ the number of ways to tile a $1 \times (n+r)$ grid using white tiles of any length (whose total length is $n$) and $r$ indistinguishable red squares, i.e. tiles of length $1$.
	Such a tiling is called a \emph{two-toned tiling of length $n+r$}, or simply an \emph{$(n+r)$-tiling} when it is understood that the tiling is two-toned.
\end{defn}

If $n=0$, then $a(r,0)$ corresponds to a tiling using just the indistinguishable red squares, and $a(r,0) = 1$.
If $r=0$, then $a(0,n)$ corresponds to a tiling by just the white tiles of lengths $1$ to $n$, hence $a(0,n)$ is the number of compositions of $n$, i.e. $a(0,n) = 2^{n-1}$ for $n \geq 1$.
Values of $a(r,n)$ for small choices of $r$ and $n$ are displayed in Table~\ref{tab: a(r,n)}, along with known OEIS \cite{oeis} sequences.

\begin{table}
	\[\begin{array}{|c|c|c|c|c|c|c|c|}
		\hline
		$\backslashbox{$r$}{$n$}$ & 0 & 1 & 2 & 3 & 4 & 5 & \text{OEIS sequence }\\  \hline
		0 & 1 & 1 & 2 & 4 & 8 & 16 & \seqnum{A001782} \\ \hline
		1 & 1 & 2 & 5 & 12 & 28 & 64 & \seqnum{A045623} \\ \hline
		2 & 1 & 3 & 9 & 25 & 66 & 168 & \seqnum{A058396} \\ \hline
		3 & 1 & 4 & 14 & 44 & 129 & 360 & \seqnum{A062109} \\ \hline
		4 & 1 & 5 & 20 & 70 & 225 & 681 & \seqnum{A169792} \\ \hline
		5 & 1 & 6 & 27 & 104 & 363 & 1182 & \seqnum{A169793} \\
		\hline
	\end{array}\]
	\caption{The numbers $a(r,n)$ for small values of $r,n$.
			The OEIS entry in row $i$ corresponds to the sequence $\{a(i,n)\}_{n \geq 0}$.}\label{tab: a(r,n)}
\end{table}

Using the red squares can allow one to compute the number of compositions of an integer where there are a specified number of specified parts.
For example, $a(1,n)$ can denote the number of compositions of $n+k$ such that exactly one part is $k$ and all remaining parts are at most $n$.
Red squares can even represent a specified run of integers as long as none of the integers in the run are already represented by one or more of the white tiles.
Indeed, Table~\ref{tab: notation} lists a variety of applications of two-toned tilings and related notation that will be addressed throughout this article.

\begin{table}
\begin{tabular}{|c|l|} \hline
	Notation & Meaning \\ \hline
	$a(r,n)$ & Number of two-toned tilings of length $r+n$ \\
	$a(r,n,k)$ & Number of ways to tile a $1 \times (n+r)$ grid using white tiles of \\
			& \quad lengths $1$ to $k$ (total length $n$) and $r$ indistinguishable red squares \\
	$a_s(r,n)$ & The number of two-toned tilings of length $r+n+s$ \\
	$C_a(r,n)$ & Number of parts of all compositions counted by $a(r,n)$ \\
	$C_b(n,k)$ & Number of compositions of $n$ where all parts $k$ occur consecutively \\
	$C_S(n)$ & Number of compositions of $n$ with parts from $S \subseteq [n]$ \\
	$C(n,\widehat k)$ & Number of compositions of $n$ with no parts $k$ \\
	$C(n,m,\widehat k)$ & Number of $(n+m)$-tilings using white tiles of any length except $k$\\
	$C(n,[k])$ & Number of compositions of $n$ with no parts a multiple of $k$ \\
	$CF(n,k)$ & Number of compositions of $n$ with $k$ frozen parts \\
	$CF(n,[k])$ & Number of compositions of $n$ having multiples of $k$ frozen \\
	$C(n, \langle k_1,\dots,k_m\rangle)$ & Number of compositions of $n$ using only parts $k_1,\dots,k_m$ \\
	$E_p(n,k)$ & Number of compositions of $n$ with exactly $p$ parts $k$ \\ 
	$E_p(n,m,k)$ & Number of compositions of $n$ with parts at most $k$ having exactly\\
			& \quad $p$ parts $m$ \\
	$E(n)$ & Total number of parts over all compositions of $n$\\
	$F(n,k,r)$ & The $r^{th}$ convolution of $\{F(j,k)\}_{j=1}^n$ \\
	$G(n,k)$ & Number of compositions of $n$ with largest part $k$ \\ 
	$G(n,k,r)$ & Number of compositions of $n$ with largest part $k$ appearing \\
			& \quad exactly $r$ times \\
	$L(n,k)$ & Number of compositions of $n$ that have at least one instance of \\
			& \quad $k$ as a part \\
	$L(n,m,k)$ & Number of compositions of $n$ with parts at most $k$ having at least \\
			& \quad one part $m$\\ 
	$L_p(n,m,k)$ &  Number of compositions of $n$ with parts at most $k$ having at least \\
			&\quad $p$ parts $k$\\
	$m(r,n)$ & Number of $(n+r)$-tilings when $r$ red squares combine \\
			&\quad palindromically with the palindromic white tile arrangements\\ 
	$\negF(n,k)$ & The $n^{th}$ negatively-indexed $k$-step Fibonacci number; see \\
			& \quad Definition~\ref{def: negF} \\
	$R(n)$ & Number of runs in all compositions of $n$ \\ 
	$R(n,k)$ & Number of runs of $k$ in all compositions of $n$ \\
	$R(n,k,l)$ & Number of runs of $k$ of length $l$ over all compositions of $n$ \\
	$r(n,\{k\})$ & Number of runs in all compositions of $n$ with largest part $k$ \\ 
	$r(n,j,\{k\})$ & Number of runs of $j$ in all compositions of $n$ having parts at most $k$\\
	$S(n,k)$ & Total parts $k$ over all compositions of $n$ \\ \hline
\end{tabular}
\caption{Collection of notation used throughout the article.}\label{tab: notation}
\end{table}


\section{Identities for $a(r,n)$}

While there are several explicit expressions for $a(r,n)$ given in \cite{twotoned,chinnsimay} the recurrence relation is among the most convenient.
By \cite[Identity 1]{twotoned}, $a(r,n)$ has the recurrence
\[
	a(r,n) = a(r-1,n) + 2a(r,n-1) - a(r-1,n-1)
\]
for $r, n > 1$ with the initial conditions
\[
	a(r,n) = 
		\begin{cases}
			1 & \text{ if } r \geq 0,\,n = 0, \\
			2^{n-1} & \text{ if } r=0,\, n \geq 1.
		\end{cases}
\]
For example, $a(5,5) = a(4,5) + 2a(5,4) - a(4,4) = 1182$. 

From the recurrence for $a(r,n)$, it is a straightforward induction argument to show that for fixed $r \geq 0$, 
\begin{equation}\label{eq:1}
	\sum_{n \geq 0} a(r,n)x^n = \left(\frac{1-x}{1-2x}\right)^{r+1}.\
\end{equation}
We can extend this generating function to consider both $r$ and $n$ by setting
\[
	A(x,y) = \sum_{r \geq 0} \sum_{n \geq 0} a(r,n)x^ry^n.
\]
It quickly follows that
\[
	A(x,y) = \sum_{r \geq 0}  \left(\frac{1-y}{1-2y}\right)^{r+1}x^r =  \frac{1-y}{1-2y-x+xy}.
\]
From these calculations, we see that for fixed $r$, $a(r,n)$ is the $r$-th convolution of the sequence of compositions of $n$.
That is, for $r > 0$, 
\[
	a(r,n) = \sum_{j = 0}^n a(r-1,n-j)a(0,j).
\]
Summations of $a(r,n)$ also have important applications, which in \cite{twotoned} motivated the following definition.

\begin{defn}
	Let $s \geq 0$.
	Denote by $a_s(r,n)$ the number of two-toned tilings of a $1 \times (n+r+s)$ grid with $r$ red squares, with the restriction that the last $s$ tiles must be white.
\end{defn}

\begin{table}
	\[\begin{array}{|c|c|c|c|c|c|c|}
		\hline
		$\backslashbox{$s$}{$n$}$ & 0 & 1 & 2 & 3 & 4\\  \hline
		0 & 1 & 3 & 9 & 25 & 66 \\ \hline
		1 & 1 & 4 & 13 & 38 & 104 \\ \hline
		2 & 1 & 5 & 18 & 56 & 160 \\ \hline
		3 & 1 & 6 & 24 & 80 & 240 \\ \hline
		4 & 1 & 7 & 31 & 111 & 351 \\ \hline
		5 & 1 & 8 & 39 & 150 & 501 \\ \hline
		6 & 1 & 9 & 48 & 198 & 699 \\ \hline
		7 & 1 & 10 & 58 & 256 & 955 \\ \hline
		8 & 1 & 11 & 69 & 325 & 1280 \\
		\hline
	\end{array}\]
	\caption{The numbers $a_s(2,n)$ for small values of $s,n$.}\label{tab: a_s(2,n)}
\end{table}

\begin{table}
	\[\begin{array}{|c|c|c|c|c|c|c|c|} \hline
		a_0(0,n) & a_1(1,n) & a_2(2,n) & a_3(3,n) & a_4(4,n) & a_5(5,n) & a_6(6,n) & a_7(7,n) \\ \hline
		\seqnum{A058396} & \seqnum{A049611} & \seqnum{A001793} & \seqnum{A001788} & \seqnum{A055580} & \seqnum{A055581} & \seqnum{A055582} & \seqnum{A055583} \\ \hline
	\end{array}\]
	\caption{OEIS sequences for $a_r(r,n)$ when $r \leq 7$.}
\end{table}
Note that $a_0(r,n) = a(r,n)$.
Values of $a_s(2,n)$ for small values of $s$ and $n$ are given in Table~\ref{tab: a_s(2,n)}, and OEIS entries for $\{a_r(r,n)\}_{n\geq0}$ for small $r$ are listed in Table~\ref{tab: a_r(r,n)}.
By \cite[Identity 6]{twotoned}, we also have
\[
	a_s(r,n) = \sum_{i = 0}^n a_{s-1}(r,i),
\]
from which
\[
	\sum_{n \geq 0} a_s(r,n)x^n = \frac{1}{(1-x)^s}\left(\frac{1-x}{1-2x}\right)^{r+1}
\]
follows.
If $r = s$, then
\[
	\sum_{n \geq 0} a_r(r,n)x^n = \frac{1-x}{(1-2x)^{r+1}},
\]
from which we can algebraically extract
\[
	a_r(r,n) = 2^{n-1}\left(\binom{n+r}{r} + \binom{n+r-1}{r-1}\right),
\]
which is consistent with \cite[Identity 9]{twotoned}.
If $s= r+1$, then
\[
	\sum_{n \geq 0} a_{r+1}(r,n)x^n = \frac{1}{(1-2x)^{r+1}},
\]
from which we can also algebraically extract
\[
	a_{r+1}(r,n) = 2^n\binom{n+r}{r}.
\]

Now, by \cite[Identity 6]{twotoned}, we may write a new recurrence for $a(r,n)$:
\[
	a(r,n) = a_1(r,n-1) + a(r-1,n).
\]
Using \cite[Identity 8]{twotoned}, we then conclude that
\[
	a_s(r,n) = \sum_{j=0}^n \binom{n-1+s}{j-1+s}\binom{r+j}{r}.
\]
The following conjectured identity has a similar flavor to these identities, but a proof has remained elusive.

\begin{conj}
	For all $s,r, n \geq 1$, we have
	\[
		a_s(r,n) = 2^{n-r-1+s}\sum_{j = 0}^{r+1-s} \binom{r+1-s}{j}\binom{n+r-j}{n}.
	\]
\end{conj}
Note that, as an immediate consequence of this conjecture, setting $s = 0$ recovers \cite[Identity 5]{twotoned}, that is,
\[
	a(r,n) = 2^{n-r-1}\sum_{j=0}^r \binom{r+1}{j}\binom{n+r-j}{n}.
\]


\section{Applications}

In this section we will present several applications of $(n+r)$-tilings to compositions of $n$ with various restrictions.
These applications all make use of the function $a_s(r,n)$ in some manner, either explicitly or implicitly. 

\subsection{$k$-step Fibonacci numbers, positively and negatively indexed}

\begin{defn}
	The \emph{$n^{th}$ $k$-step Fibonacci number}, denoted $F(n,k)$, is defined as
	\[
		F(n,k) = 
			\begin{cases}
				0 & \text{ if } n \leq 0,\\
				1 & \text{ if } n = 1,\\
				\sum_{j=1}^k F(n-j,k) & \text{ if } n \geq 2.
			\end{cases}
	\]
\end{defn}

For example, the $3$-step Fibonacci numbers begin $\dots, 0, 0, 1, 1, 2, 4, 7, 13, 24, 44, 81, 149,\dots$.
Of course, setting $k=2$ recovers the usual Fibonacci numbers.
It is easy to verify that $F(j,k) = 2^{j-2}$ for $2 \leq j \leq k$.

Benjamin et. al \cite[Identity 10]{twotoned} gave a combinatorial proof of the following result involving the $k$-step Fibonacci numbers.
Here, we provide a generating function proof.

\begin{thm}\label{thm: thm1}
	For $k \geq 0$ and $n \geq -1$,
	\[
		F(n+1,k)  = \sum_{j \geq 0} (-1)^ja_j(j,n-j(k+1)).
	\]
\end{thm}

\begin{proof}
	We know that the generating function for $F(n+1,k)$ is $(1-x-x^2-\cdots - x^k)^{-1}$.
	So, we examine the generating function for the right side of the desired equation: by reindexing, we get
	\[\begin{aligned}
		\sum_{m \geq 0} \sum_{j \geq 0} (-1)^ja_j(j,m-j(k+1))x^m =& \sum_{l \geq 0} (-1)^l\frac{1-x}{(1-2x)^{l+1}}x^{l(k+1)} \\
			=& \frac{1-x}{1-2x}\sum_{l \geq 0} \left(\frac{-x^{k+1}}{1-2x}\right)^l \\
			=& \frac{1-x}{1-2x}\left(\frac{1}{1+\frac{x^{k+1}}{1-2x}}\right).
	\end{aligned}\]
	From here, routine elementary algebra simplification shows that the resulting generating function is the same as that of $F(n+1,k)$.
\end{proof}

From the expression for $a_r(r,n)$, Benjamin et al. provided an explicit formula \cite[Identity 9]{twotoned} for $F(n+1,k)$.
For $n,k \geq 1$,
\[
	F(n+1,k) = \sum_{r - 0}^{\lfloor n/(k+1) \rfloor} (-1)^r\binom{n-rk}{r}\frac{n-rk+r}{n-rk}2^{n-rk-r-1}.
\]
Thus, $a_r(r,n)$ has the following recurrence relation.

\begin{prop}
	For $r,n \geq 1$,
	\[
		a_r(r,n) = 2a_r(r,n-1) + a_{r-1}(r-1,n).
	\]
\end{prop}

\begin{proof}
	We know that the generating function for $a_r(r,n)$ can be written as
	\[
		\frac{(1-x)^{r+1}}{(1-2x)^{r+1}(1-x)^r}.
	\]
	So, the generating function of the right hand side of the desired identity is
	\[
		2x\frac{(1-x)^{r+1}}{(1-2x)^{r+1}(1-x)^r} + \frac{(1-x)^{r}}{(1-2x)^{r}(1-x)^{r-1}}.
	\]
	Routine algebraic simplifications show that this reduces to the generating function for $a_r(r,n)$, as desired.
\end{proof}

Values of $a_r(r,n)$ for small $r,n$ are displayed in Table~\ref{tab: a_r(r,n)}.
Soon we will see that the sequences along the diagonals in the table (an example of which is bolded) arise in an interesting manner.

\begin{rmk}
	We pause here to identify a number of connections between the numbers $a_s(r,n)$ and other results in the literature.
	It appears that the rows of Table~\ref{tab: a_s(2,n)} are related to \emph{$p$-ascent sequences} as defined in \cite{ascentsequences}.
	Namely, $\{a_1(2,n)\}_{n \geq 0}$ appears to coincide with $3$-ascent sequences; see OEIS sequence \seqnum{A049611}.
	The sequence $\{a_4(2,n)\}_{n \geq 0}$ appears to be a Bj\"orner-Welker sequence \cite{bjornerwelker}, providing Betti numbers of certain manifolds.
	The sequence $\{a_6(2,n)\}_{n \geq 0}$ appears to give the \emph{popularity} of the pattern $231$ in permutations of $[n]$; see \cite{popularity} and OEIS sequence \seqnum{A055581}.
	
	In Table~\ref{tab: a_r(r,n)}, $\{a_1(1,n)\}_{n \geq 0}$ arises in \cite[Proposition 41]{countingfunction} as $\binom{1,n}{1,2}$, which counts what they call \emph{$(n+1)$-insets} of a certain set $X$.
	It also appears that $\{a_3(2,n)\}_{n \geq 0}$ of the previous tbale arises as $\binom{0,n}{k,2}$ in the same notation.
	Notably, all of Table~\ref{tab: a_r(r,n)} appears to be exactly the unsigned coefficients Chebyshev polynomials of the first kind; see OEIS sequence \seqnum{A081277} and \cite{chebyshev}.
	We invite the reader to establish bijections between two-toned tilings (and other objects studied within this article) and the objects listed above.
\end{rmk}

\begin{table}
\begin{tabular}{|c|c|c|c|c|c|c|c|} \hline
	\backslashbox{$r$}{$n$} & 0 & 1 & 2 & 3 & 4 & 5 & 6 \\ \hline
	0 & 1 & 1 & 2 & {\bf 4} & 8 & 16 & 32 \\ \hline
	1 & 1 & 3 & {\bf 8} & 20 & 48 & 112 &  \\ \hline
	2 & 1 & {\bf 5} & 18 & 56 & 160 & & \\ \hline
	3 & {\bf 1} & 7 & 32 & 120 & & & \\ \hline
	4 & 1 & 9 & 50 & & & & \\ \hline
	5 & 1 & 10 &  &  &  &  & \\ \hline
	6 & 1 &  &  &  & & &  \\ \hline
\end{tabular}
\caption{Values of $a_r(r,n)$ for small choices of $r$ and $n$.}\label{tab: a_r(r,n)}
\end{table}

We will now consider a subtle but significant variation of the $k$-step Fibonacci numbers.

\begin{defn}\label{def: negF}
	For a positive integer $k$, the \emph{negatively-indexed $n^{th}$ $k$-step Fibonacci numbers} is
	\[
		\negF(n,k) = \negF(n-1,k) + \cdots + \negF(n-k,k)
	\]
	for \emph{any} $n$, using the initial conditions $\negF(n,k) = 0$ for $n = 0,-1,-2,\dots,-(k-2)$ and $\negF(1,k) = 1$ for all $k$.
\end{defn}

So, we have $\negF(n,k) = F(n,k)$ whenever $n$ is nonnegative, but $\negF(n,k)$ may be nonzero for negative values of $n$.
For example, with $k=3$, a portion of the values $\negF(n,3)$ is
\[
	\begin{array}{c|ccccccccccc}
		n & -9 & -8 & -7 & -6 & -5 & -4 & -3 & -2 & -1 & 0 & 1 \\ \hline
		\negF(n,3) & -8 & 4 & 1 & -3 & 2 & 0 & -1 & 1 & 0 & 0 & 1 \\
	\end{array}
\]

\begin{thm}
	For integers $n,k \geq 1$, let $n+1 = km+r$ where $0 \leq r < k$.
	Then
	\[
		\negF(-(n+1),k) = \sum_{j \geq 0} (-1)^{r-jk}a_{r+jk}(r+jk,m-r-j(k+1)).
	\]
\end{thm}

\begin{proof}
	We approach this problem by defining the sequence $\{b_i\}_{i \geq 0}$ by setting $\negF(n,k) = b_{-n+1}$ and finding the generating function for $\{b_i\}_{i\geq 0}$.
	In this formulation, our new sequence satisfies
	\begin{enumerate}
		\item $b_0 = 0$,
		\item $b_1 = \cdots = b _{k-1} = 0$, and
		\item $b_i = b_{i-k} - \sum_{m=1}^{k-1} b_{i-m}$ for $i \geq k$.
	\end{enumerate}
	Through standard algebraic arguments, we obtain
	\[
		\sum_{i \geq 0} b_ix^i = \frac{1-x^k}{1-2x^k+x^{k+1}}.
	\]
	
	By directly computing the generating function of the right hand side of the desired equality, we obtain the same expression. 
	The claim then follows.
\end{proof}

Note that positively-indexed classical Fibonacci numbers, i.e. when $k=2$, obey the equation
\[
	F(n+1,2) = \sum_{i \geq 0} (-1)^i a_i(i,n-3i),
\]
while negatively-indexed Fibonacci numbers obey the equations
\[
	\negF(-(n+1),2) = \sum_{i \geq 0} a_{2i}(2i,m-3i)
\]
if $n = 2m-1$, and
\[
	\negF(-(n+1),2) = \sum_{i \geq 0} (-1)^{i+1}a_{2i+1}(2i+1,m-(3i+1))
\]
if $n = 2m$.
Putting these together, we see that $\negF(n,2) = F(-n,2)$ when $n < 0$. 
However, as $k$ increases, $F(n,k)$ approaches $2^{n-1}$ for each $n$, while in $\negF(n,k)$ for $n < 0$, the sequences $\{a_{j-i}(j-i,i)\}_{i=0}^j$ appear for each $j \geq 1$, separated by increasingly-large strings of $0$s. 
The strings of consecutive nonzero integers are exactly those that are on the diagonals in Table~\ref{tab: a_r(r,n)}.

\subsection{Convolutions of the $k$-step Fibonacci sequence applied to compositions}

Recall that if $\{s_i\}_{i \geq 0}$ is a sequence of real numbers, its \emph{$r^{th}$ convolution} is the sequence of coefficients of
\[
	\left(\sum_{k \geq 0} s_kx^k\right)^{r+1}
\]
expressed in the standard vector space basis $1,x,x^2,\dots$ over $\RR$.
Let $\{F(n,k,r)\}_{n \geq 0}$ denote the $r$-th convolution of the $k$-step Fibonacci sequence $\{F(j,k)\}_{j=1}^n$, that is, $F(n,k,r)$ is the coefficient of $x^k$ in 
\[
	\left(\sum_{i = 0}^n F(j,k)x^j\right)^{r+1}. 
\]

To illustrate, 
\[\begin{aligned}
	F(n,k,0) &= F(n,k), \\
	F(n,k,1) &= \sum_{j=1}^n F(n+1-j,k)F(j,k), \text{ and } \\
	F(n,k,r) &= \sum_{j=1}^n F(n+1-j,k,r-1)F(j,k,r-1).
\end{aligned}\]

Obtaining expressions for $F(n+1,k,r)$ with $k \geq 1$ motivates the following definition.
\begin{defn}
	For nonnegative integers $r$ and $n$ and $1 \leq k \leq n$, let $a(r,n,k)$ denote the number of ways to tile a $1 \times (n+r)$ grid using white tiles of lengths $1$ to $k$ (with total length $n$) and $r$ indistinguishable red squares.
	Moreover, define
	\[
		a_s(r,n,k) = \sum_{j=0}^n a_{s-1}(r,j,k)
	\]
	where $a_0(r,n,k) = a(r,n,k)$ for all $r,n,k$. 
\end{defn}

Note that if $n=0$, then $a(r,0,k)$ corresponds to a tiling of just the indistinguishable red squares, so that $a(r,0,k) = 1$.
If $r=0$, then $a(0,n,k)$ corresponds to a tiling of just the white tiles of lengths $1$ to $k$, which we know to be $F(n+1,k)$;
note how, as $k \to \infty$, we recover $a(0,n) = C(n)$, the number of compositions of $n$.

Now, from the definition of two-toned tilings, we directly have the following recurrence.

\begin{prop}\label{prop: rec}
	For all $n,k,r \geq 0$, 
	\[
		a(r,n,k) = \sum_{j=1}^k a(r,n-j,k).
	\]
\end{prop}

Furthermore, we will extend equation \eqref{eq:1}.
By a routine generating function exercise from Proposition~\ref{prop: rec}, we obtain the following.

\begin{prop}\label{prop: f_r}
	For all $r$ and $k$ we have the generating function
	\[
		\sum_{n \geq 0} a(r,n,k)x^n = \left(\frac{1-x}{1-2x+x^{k+1}}\right)^{r+1}.
	\]
\end{prop}

It is routine to show, for example by \cite[Chapter 2, Rule 3]{generatingfunctionology}, that the generating function $f_{r,k}(n)$ in the proof of Proposition~\ref{prop: f_r} is the same as the generating function for $\{F(n+1,k,r)\}_{n \geq 0}$.
This gives us the following corollary.

\begin{cor}
	For all $n,k,r \geq 0$, we have $a(r,n,k) = F(n+1,k,r)$.
\end{cor}

Although the recurrence relation for $a(r,n,k)$ given in Proposition~\ref{prop: rec} provides the most efficient method of computation known thus far, $a(r,n,k)$ can also be expressed as an alternating sum of summed $a_s(r,n)$ quantities.

\begin{prop}\label{prop: F = a}
	For all $n,k,r \geq 0$,
	\[
		F(n+1,k,r) = a(r,n,k) = \sum_{j \geq 0} (-1)^j\binom{r+j}{r}a_j(r+j,n-j(k+1)).
	\]
\end{prop}

\begin{proof}
	Note that the generating function for the right side of the identity is
	\[
		\sum_{n \geq 0} \left(\sum_{j \geq 0} (-1)^j\binom{r+j}{r}a_j(r+j,n-j(k+1))\right)x^n,
	\]
	which simplifies to
	\[
		\sum_{i \geq 0} (-1)^i\binom{r+1}{i}\left(\frac{1-x}{1-2x}\right)^{r+1+i}\frac{x^{i(k+1)}}{(1-x)^i}.
	\]
	Continuing, we get
	\[\begin{aligned}
			\left(\frac{1-x}{1-2x}\right)^{r+1}\sum_{i \geq 0} (-1)^i\binom{r+1}{i}\left(\frac{1-x}{1-2x}\right)^i\frac{x^{i(k+1)}}{(1-x)^i} &= \left(\frac{1-x}{1-2x}\right)^{r+1}\frac{1}{\left(1+\frac{(1-x)x^{k+1}}{(1-2x)(1-x)}\right)^{r+1}} \\
			&= \left(\frac{1-x}{1-2x+x^{k+1}}\right)^{r+1},
	\end{aligned}\]
	This is exactly the generating function for $F(n+1,k,r)$.
\end{proof}

In \cite{twotoned}, the $a(r,n)$ functions were used to compute the number of compositions of $n$ with least or exactly $p$ instances of a given part. 
The convoluted $k$-step Fibonacci sequences allow these quantities to be computed when the parts are at most $k$.

\begin{defn}
	For positive integers $m,n,k$ with $m \leq k$, let $L(n,m,k)$ denote the number of compositions of $n$ with parts at most $k$ such that at least one part is $m$. 
\end{defn}

Table~\ref{tab: F(n+1,3,r)} gives values of $F(n+1,3,r)$ for small choices of $n,r$.
We note that the sequence $\{F(4,3,r)\}_{r \geq 0}$ has multiple existing combinatorial interpretations, described within the OEIS entry \seqnum{A000096}.

\begin{table}
\begin{tabular}{|c|c|c|c|c|c|c|c|c|c|} \hline
	\backslashbox{$j$}{$n$} & 1 & 2 & 3 & 4 & 5 & 6 & 7 & 8 & 9 \\ \hline
	0 & 1 & 1 & 2 & 4 & 7 & 13 & 24 & 44 & 81 \\ \hline
	1 & 1 & 2 & 5 & 12 & 26 & 56 & 118 & 244 &  \\ \hline
	2 & 1 & 3 & 9 & 25 & 63 & 153 & 359 & & \\ \hline
	3 & 1 & 4 & 14 & 44 & 125 & 336 & & & \\ \hline
	4 & 1 & 5 & 20 & 70 & 220 & & & & \\ \hline
	5 & 1 & 6 & 27 & 104 &  &  & & & \\ \hline
	6 & 1 & 7 & 35 &  & & &  & & \\ \hline
	7 & 1 & 8 &  &  & & &  & & \\ \hline
	8 & 1 &  &  &  & & &  & & \\ \hline
\end{tabular}
\caption{Values of $F(n+1,3,r)$ for small choices of $n,r$.}\label{tab: F(n+1,3,r)}
\end{table}

\begin{prop}
	For all $m,n,k$ with $m \leq k$,
	\[
		L(n,m,k) = \sum_{j \geq 1} (-1)^{j-1}F(n+1-jm,k,j).
	\]
\end{prop}

\begin{proof}
	First fix a positive integer $j$ and consider $F(n+1-jm,k,j)$.
	By Proposition~\ref{prop: F = a}, $F(n+1-jm,k,j) = a(j,n-jm,k)$, the right side of which counts the number of $((n-jm)+j)$-tilings whose white tiles have length at most $k$.
	By replacing each red square with a pink tile of length $m$, the resulting tiling has length $(n-jm) + jm  = n$.
	Each of these correspond to a composition of $n$ with parts at most $k$ where at least one part has length $m$. 
	The result follows from applying inclusion-exclusion.
\end{proof}

Next we consider a refinement of $L(n,m,k)$.

\begin{defn}
	For $p,k,n \geq 1$ with $m \leq k$, let $L_p(n,m,k)$ denote the number of compositions of $n$ with parts at most $k$ such that there are at least $p$ parts $m$. 
\end{defn}

With this notation, $L_1(n,m,k) = L(n,m,k)$.
Further, the next proposition follows from the proof of \cite[Identity 12]{twotoned}.

\begin{prop}
For all $p,k,n \geq 1$ with $m \leq k$,
\[
	L_p(n,m,k) = \sum_{j \geq p} (-1)^{j-p}\binom{j-1}{p-1}F(n+1-jm,k,j).
\]
\end{prop}

Next, we consider compositions having exactly a particular number of parts.

\begin{defn}
	For $p,m,n,k \geq 1$ and $m \leq k$, let $E_p(n,m,k)$ denote the number of compositions of $n$ with parts at most $k$ having exactly $p$ parts $k$.
\end{defn}

\begin{prop}
For $p,m,n,k \geq $ with $m \leq k$,
\[
	E_p(n,m,k) = \sum_{j \geq p} (-1)^{j-p}\binom{j}{p}F(n+1-jm,k,j).
\]
\end{prop}

\begin{proof}
	The equality follows from the observation
	\[
		E_p(n,m,k) = L_p(n,m,k) - L_{p+1}(n,m,k)
	\]
	and from Pascal's identity for binomial coefficients.
\end{proof}

We note that the previous proposition was also implicitly established in the proof of \cite[Identity 13]{twotoned}.

Let $S(n,k)$ denote the total number of parts $k$ over all compositions of $n$.
For example, $S(2,4) = 5$ since there are three compositions of $4$ in which $2$ appears once and one composition in which $2$ appears twice. 

\begin{prop}
	For $1 \leq k < n$, $S(n,k) = 2^{n-2}(n+1)$.
\end{prop}

\begin{proof}
It has been shown \cite{cuisinaire} that
\begin{equation}\label{eq: 4.2-1}
	S(n,k) = 2^{n-k-2}(n-k+3).
\end{equation}
By \cite[Identity 9]{twotoned}, $a_1(1,n) = 2^{n-2}(n+1)$, and hence
\begin{equation}\label{eq: 4.2-2}
	S(n,k) = a(1,n-k).
\end{equation}
With this in mind, \cite[Identity 13]{twotoned} provides a route for establishing \eqref{eq: 4.2-2} and therefore \eqref{eq: 4.2-1}.
The number of times $k$ is a part in compositions having exactly $p$ copies of $k$ is $kE_p(n,k)$.
So,
\[
	S(n,k) = \sum_{j \geq 1} jE_j(n,k).
\]
Applying \cite[Identity 13]{twotoned} to each instance of $E_j(n,k)$, we get 
\[
	S(n,k) = \sum_{j=1}^n a(1,n-j) = a_1(1,n-1) = 2^{n-2}(n+1),
\]
as desired.
\end{proof}

For the next definition, let $\lambda = (\lambda_1,\dots,\lambda_m)$ be an integer composition.
A \emph{run} in $\lambda$ is a subsequence $\lambda_i,\lambda_{i+1},\dots,\lambda_{i+l}$ such that 
\[
	\lambda_{i-1} \neq \lambda_i = \lambda_{i+1} = \cdots = \lambda_{i+l} \neq \lambda_{i+l+1},
\]
using the convention $\lambda_0 = \lambda_{m+1} = 0$. 
The \emph{length} of the run is $l$.
So, for example, $(2,2,2,4,1,1,2)$ has four runs: one of length three, one of length two, and two of length one. 

\begin{defn}
	For positive integers $j,k,n$ with $j \leq k \leq n$, let $r(n,j,\{k\})$ denote the number of runs of $j$, irrespective of length of the run, in all compositions of $n$ whose parts are at most $k$.
	Further, let $r(n,\{k\})$ denote the total number of runs over all compositions of $n$ whose parts are at most $k$.
\end{defn}

\begin{prop}\label{prop: X3}
	For $1 \leq j \leq k \leq n$,
	\[
		r(n,j,\{k\}) = F(n+1-j,k,1) - F(n+1-2j,k,1).
	\]
\end{prop}

\begin{proof}
	If $j > n-j$, then all runs of $j$ are of length one. 
	The part $j$ may then be represented by a red square in a tiling of $n$, and is combined with white squares of lengths less than $j$.
	The number of such tilings is given by $a(1, n-j, k)$, which is equal to $F(n+1-j, k, 1)$.
	Also note that if $j > n-j$, then $a(1, n-2j, k) = 0$.

	If $j \leq n-j$, then some of the white tiles have length $j$, thereby being otherwise indistinguishable from the red square representing $j$.
	Therefore, some of the $n$-tilings merely increase the length of existing runs.
	The number of such instances is $a(1,n-2j, k) = F(n+1-2j, k, 1)$ since increasing the length of a run has the same effect. 
	Hence, $r(n, j,{k}) = F(n+ 1-j, k,1)-F(n+ 1-2j, k,1)$.
\end{proof}

\begin{prop}
	For all $1 \leq k \leq n$,
	\[
		r(n,\{k\}) = \sum_{j \geq 0} F(n-2j,k,1).
	\]
\end{prop}

\begin{proof}
	By definition,
	\[
		r(n,\{k\}) = r(n,1,\{k\}) + \cdots + r(n,k,\{k\}).
	\]
	By Proposition~\ref{prop: X3},
	\[
		r(n,j,\{k\}) = F(n+1-j,k,1) - F(n+1-2j,k,1).
	\]
	Thus,
	\[
		r(n,\{k\}) = \sum_{j \geq 0} (F(n-j,k,1) - F(n-1-2j,k,1),
	\]
	which simplifies to the desired sum.
\end{proof}

\begin{defn}
	Denote by $C(n,\widehat k)$ the number of compositions of $n$ for which $k$ is not a part. 
\end{defn}

From \cite[Theorem 1, Equation (1)]{ChinnHeubach}, we have the following recurrence:
\[
	C(n, \widehat k) = 2C(n-1,\widehat k) + C(n-k-1,\widehat k) - C(n-k,\widehat k).
\]
Consequently, the generating function for $C(n,\widehat k)$ for fixed $k$ is
\[
	\frac{1-x}{1-2x+x^k-x^{k+1}}.
\]

Recall that $C(n)$ denotes the number of compositions of $n$.
Let $L(n,k)$ denote the number of compositions of $n$ that have at least one instance of $k$ as a part.
It follows immediately that
\[
	C(n,\widehat k) = C(n) - L(n,k).
\]
But $C(n) = a(0,n)$ and \cite[Identity 11]{twotoned} gives
\[
	L(n,k) = \sum_{j \geq 1} (-1)^{j-1}a(j,n-jk)
\]
for $n,k \geq 1$.
This leads to the following result directly.

\begin{prop}\label{prop: 4-1}
	For $n,k \geq 1$,
	\[
		C(n,\widehat k) = \sum_{j \geq 0} (-1)^ja(j,n-jk).
	\]
\end{prop}

Let $S = \{s_1,\dots,s_m\}$ be a subset of $[n]$, and let $C_S(n)$ denote the number of compositions of $n$ whose unique parts are in $S$. 
Then the generating function for $C_S(n)$ is
\[
	\sum_{n \geq 0} C_S(n)x^n = \frac{1}{1-x^{s_1} - \cdots - x^{s_m}}.
\]
From the above equation, it follows that the generating function for $C(n,\widehat k)$ is
\[
	\frac{1}{1+x^k - \sum_{i \geq 0} x^i}.
\]

\begin{thm}
	Recall that $F(n-1,2)$ is the $(n-1)$-st (classical) Fibonacci number. Then
	\[
		F(n-1,2) = \sum_{j \geq 0} (-1)^ja(j,n-j).
	\]
\end{thm}

\begin{proof}
	This follows by setting $k=1$ in Proposition~\ref{prop: 4-1}.
\end{proof}

Note that if $k > n/2$, then 
\[
	C(n,\widehat k) = a(0,n) - a(1,n-k) = 2^{n-1} - 2^{n-k}(n-k+3).
\]
For $r,n,k \geq 0$, let $C(n,m,\widehat k)$ denote the number of $(n+m)$-tilings using white tiles of any length except $k$.

\begin{thm}
	For $n,m,k \geq 0$, 
	\[\begin{aligned}
		C(n,m,\widehat k) &= E_m(n+mk,k) \\
			&= \sum_{j \geq m} (-1)^{j-m}\binom{j}{m}a(j,n-k(j-m)).
	\end{aligned}\]
\end{thm}

\begin{proof}
	Consider $C(n,m,\widehat k)$. 
	The white tiles correspond to tiles of any length except $k$.
	The red tiles cannot correspond to a nonnegative part already represented by a white tile, but they can correspond to $k$.
	So, there is a bijection between the two-toned tilings of $C(n, m, \widehat{k})$ and the number of compositions of $n+mk$ having exactly $m$ copies of $k$, which is given by $E_m(n+mk,k)$. 
	The summation arises from substituting $m$ for $p$ and $n+mk$ for $n$.
\end{proof}

To illustrate this result, we have the following special cases:
\[\begin{aligned}
	C(n,0,\widehat k) &= a(0,n) - a(1,n-k) + a(2,n-2k) - \cdots \\
	C(n,1,\widehat k) &= a(1,n) - a(2,n-k) + 3a(3,n-2k) - \cdots \\
	C(n+2k,2,\widehat k) &= a(2,n) - 3a(3,n-k) + 6a(4,n-2k) - \cdots
\end{aligned}\]

\subsection{The largest and second-largest parts of compositions of $n$}

Let $G(n,k)$ denote the number of compositions of $n$ having $k$ as the largest part, and let $C(n,\langle 1,\dots,k\rangle)$ denote the number of compositions of $n$ whose parts are at most $k$.
It is clear that 
\[
	C(n,\langle 1,\dots,k\rangle) - C(n,\langle 1,\dots,k-1\rangle) = G(n,k).
\]
Further, by conditioning on the final part, we get
\[
	C(n,\langle 1,\dots,k\rangle) = \sum_{j = 1}^k C(n-j,\langle 1,\dots,k\rangle),
\]
from which it follows that $\sum_{n \geq 0} C(n,\langle 1,\dots,k\rangle)x^n = \frac{1-x}{1-2x+x^{k+1}}$.
Therefore, $C(n,\langle 1,\dots,k\rangle) = F(n+1,k)$, and
\[
	G(n,k) = F(n+1,k) - F(n+1,k-1).
\]
It follows that
\[\begin{aligned}
	\sum_{n \geq 0} G(n,k)x^n &= \frac{x^{k-1}}{(1-x-\cdots -x^k)(1-x-\cdots -x^{k-1})} \\
		&= \frac{x^{k-1}(1-x)^2}{(1-2x+x^{k+1})(1-2x+x^k)}.
\end{aligned}\]
Thus, $G(n+k-1,k)$ is the convolution of $F(n+1,k)$ and $F(n+1,k-1)$.

The next result relates $F(n,k)$ to the $r$th convolutions of $F(n,k-1)$.

\begin{prop}\label{prop: thm 2-1}
	For all $n,k$,
	\[
		F(n,k) = \sum_{j \geq 0} F(n-jk,k-1,j).
	\]
\end{prop}

\begin{proof}
	By considering the generating function of the right side, we get
	\[\begin{aligned}
		\sum_{m \geq 0}\left(\sum_{j \geq 0} F(n-jk,k-1,j)\right)x^m &= \sum_{m \geq 0} \frac{x^{mk}}{(1-x-x^2-\cdots-x^{k-1})^{m+1}} \\
			&= \frac{1}{1-x-x^2-\cdots-x^{k-1}}\cdot\frac{1}{1-\frac{x^k}{1-x-x^2-\cdots-x^{k-1}}} \\
			&= \frac{1}{1-x-x^2-\cdots-x^k},
	\end{aligned}\]
	which we know to be the generating function of $F(n,k)$.
\end{proof}

Extending the definition of $G$, let $G(n,k,r)$ denote the number of compositions of $n$ having $k$ as the largest part exactly $r$ times.

\begin{prop}\label{prop: thm 2-2}
	For $n,k,r \geq 0$, we have
	\[
		G(n,k,r) = F(n+1-kr,k-1,r).
	\]
\end{prop}

\begin{proof}
	The value of $G(n,k,r)$ can be represented as a two-toned tiling by concatenating an $(n-kr)$-tiling using white tiles of length at most $k-1$ with a $kr$-tiling using just red squares.
	The number of ordered tiling arrangements is therefore $a(r,n-kr,k-1)$, which is equal to $F(n+1-kr,k-1,r)$.
\end{proof}

Note that the previous result allows us to state that
\[
	G(n,k) = \sum_{r \geq 1} G(n,k,r).
\]

\subsection{Frozen parts}

\begin{defn}
	Let $1 \leq j \leq n$.
	A \emph{$j$-partially-ordered composition} of $n$ is any equivalence class of compositions of $n$ under the relation $\lambda \sim \mu$ if and only if 
	\[
		(\lambda_1,\dots,\lambda_j,\lambda_{\sigma(j+1)},\dots,\lambda_{\sigma(k)}) = \mu
	\]
	for some permutation $\sigma \in \symm_{\{j+1,\dots,k\}}$.
	The parts $\lambda_1,\dots,\lambda_j$ are called the \emph{frozen} parts of the composition. 
\end{defn}

For example, $(8,7,9,6,4,3,2,1)$ and $(8,7,9,6,1,2,3,4)$ are the same $4$-partially-ordered composition but are distinct $3$-partially-ordered compositions. 
Moreover, it is clear that $j$-partially-ordered compositions could be defined by freezing any $j$ of the parts; we only consider freezing the first $j$ parts as a matter of computational simplicity.
Notice that the $0$-partially-frozen compositions are just the ordinary compositions of $n$, while the $n$-partially-frozen compositions are just the partitions of $n$.

Recall that $CF(n,k)$ denotes the number of compositions of $n$ with part $k$ frozen, and that $C(n,\widehat k)$ denotes the number of compositions of $n$ with no part $k$.

\begin{prop}\label{prop: thm 3-1}
	For all $n$ and $k$,
	\[
		CF(n,k) = \sum_{j \geq 0} C(n-jk,\widehat k).
	\]
\end{prop}

\begin{proof}
	The compositions of $n$ can be grouped into compositions that have no copies of $k$, one copy of $k$, and so forth through $j$ copies of $k$.
	Because the $jk$ are frozen, they can be treated as being appended to the compositions of $n-jk$ having no $k$.
	The result follows.
\end{proof}

Building off of $C(n,\langle 1,\dots,k\rangle)$, given nonnegative integers $n,k_1,\dots,k_m$ with $1 \leq m \leq n$, let $C(n,\langle k_1,\dots,k_m\rangle)$ denote the number of compositions of $n$ consisting only of parts $k_1,\dots,k_m$.
For example, $C(5,\langle 1,2,5\rangle) = 9$ since the compositions are the distinct permutations of $(5)$, $(2,2,1)$, $(2,1,1,1)$, and $(1,1,1,1,1)$.
Note as well that $C(n,\langle 1,\dots,k\rangle) = F(n+1,k)$. 

\begin{prop}\label{prop: thm 3-2}
	For all $n,k$, 
	\[
		CF(n,k) = C(n,\langle1,\dots,k,2k\rangle).
	\]
\end{prop}

\begin{proof}
	From \cite[Theorem 1]{ChinnHeubach}, the generating function for $C(n,\widehat k)$ is
	\[
		\sum_{n \geq 0} C(n,\widehat k)x^n = \frac{1-x}{1-2x+x^k-x^{k+1}}.
	\]
	By Proposition~\ref{prop: thm 3-1}, we can compare the generating functions and use routine algebra to conclude that
	\[\begin{aligned}
		\sum_{n \geq 0} CF(n,k)x^n &= (1+x^k + x^{2k} + \cdots )\sum_{n \geq 0} C(n,\widehat k)x^n \\
			&= \frac{1}{1-x^k}\cdot\frac{1-x}{1-2x+x^k-x^{k+1}}\\
			&= \frac{1}{1-x-x^2-\cdots-x^k-x^{2k}}.
	\end{aligned}\]
\end{proof}

\begin{prop}\label{prop: thm 3-3}
	For all $n,k$,
	\[
		CF(n,k) = \sum_{j \geq 0} F(n+1-2kj,k,j).
	\]
\end{prop}

\begin{proof}
	By Proposition~\ref{prop: thm 3-2}, $CF(n,k)$ consists of three things: compositions having only parts $1$ through $k$; compositions with parts from $1,\dots,k,2k$; and compositions from parts $2k$ only.
	In the first case, there are $F(n+1,k) = F(n+1,k,0)$ compositions.
	The compositions of the second type can be counted by $a(j,n-2kj,k)$, corresponding to a tiling where part $2k$ has the role of the red square.
	The compositions of the third type consist of $j$ copies of $2k$ only, which itself corresponds to $j$ red squares.
	This is counted by $a(j,2kj-2kj,k) = a(j,0,k) = 1$.
	With $a(r,n,k) = F(n+1,k,r)$, the theorem follows. 
\end{proof}

\subsection{Formulas involving the parts making up the compositions of $n$}

The two-toned tiling functions provide a quick way to evaluate what might otherwise be obscure quantities.

\begin{thm}
	Let $1 \leq j \leq n$.
	If each part $j$ used in the compositions of $n$ is replaced by the compositions of $j$, then the total number of resulting compositions is $a_1(2,n-1)$. 
\end{thm}

\begin{proof}
	The number of times $j$ is used in the compositions of $n$ is $a(1,n-j)$.
	If the compositions of $j$ are substituted for $j$ then the resulting number of compositions is $2^{j-1}a(1,n-j)$, which is equal to $a(1,n-j)a(0,j)$.
	The total number of these compositions is 
	\[
		\sum_{j = 1}^n a_1(1,n-j)a(0,j).
	\]
	The corresponding generating function is 
	\[
		\left(\frac{1}{1-x}\right)\left(\frac{1-x}{1-2x}\right)^2\left(\frac{1-x}{1-2x}\right) = \frac{1}{1-x}\left(\frac{1-x}{1-2x}\right)^3,
	\]
	which corresponds to $a_1(2,n-1)$.
\end{proof}

\begin{prop}
	Let $1 \leq j \leq n$.
	If each part $j$ used in the compositions of $n$ is replaced by the parts used in the compositions of $j$, then the number of resulting parts is $a_1(3,n-1)$. 
\end{prop}

\begin{proof}
	The number of times $j$ is used in the compositions of $n$ is $a(1,n-j)$.
	If the number of parts used in the compositions of $j$ are substituted for $j$ then the resulting number of parts is $a(1,n-j)a_1(1,n-j)$.
	The total number of these parts is
	\[
		\sum_{j=1}^n a(1,n-j)a_1(1,n-j),
	\]
	which has corresponding generating function
	\[
		\left(\frac{1-x}{1-2x}\right)^2\left(\frac{1}{1-x}\right)^2\left(\frac{1-x}{1-2x}\right) = \frac{1}{1-x}\left(\frac{1-x}{1-2x}\right)^4.
	\]
	This corresponds to $a_1(3,n-1)$.
\end{proof}

\begin{defn}
	For $r,n \geq 1$, let $C_a(r,n)$ denote the number of parts of all compositions counted by $a(r,n)$.
\end{defn}

\begin{prop}
	For all $r,n$,
	\[
		C_a(r,n) = (r+1)a_1(r+1,n-1) + ra_0(r,n).
	\]
\end{prop}

\begin{proof}
	Consider first the number of red tiles needed to form the compositions of $a(r,n)$.
	Each composition of $n$ has $r$ red tiles, so the number of red tiles needed is $ra(r,n)$.
	
	Now consider the number of white tiles needed to form the compositions of $a(r,n)$.
	For a composition of $a(r,n)$ consisting of exactly $j$ white tiles, the total number of parts is $j\binom{r+j}{r}\binom{n-1}{j-1}$, and the total number of white tiles is
	\[
		\sum_{j = 1}^n j\binom{r+j}{r}\binom{n-1}{j-1} = (r+1)a_1(r+1,n-1).
	\]
\end{proof}

Next, we denote by $C_b(n,k)$ the number of compositions of $n$ where all parts $k$ must occur consecutively.  
For example, $C_b(4,1) = 7$ since the compositions counted are $(4)$, $(3,1)$, $(1,3)$, $(2,2)$, $(2,1,1)$, $(1,1,2)$, $(1,1,1,1)$. 
We will refine $C_b(n,k)$ by $C_b(n,k,p)$, which denotes the number of compositions of $n$ having exactly $p$ parts $k$, which must occur consecutively.
So, $C_b(4,1,2) = 2$ since the compositions counted are $(2,1,1)$ and $(1,1,2)$. 

\begin{prop}
	For all $n, k$, and $p$,
	\[
		C_b(n,k,p) = C(1,n-pk,\widehat k) = E_1(n-(p-1)k,k).
	\]
\end{prop}

\begin{proof}
	Consider the compositions of $n$ in which we insist the parts that are $k$ are blocked together exactly $p$ times. 
	The quantity $n-pk$ then corresponds to the parts not equal to $k$.
	Because all of the $k$s are consecutive, the compositions are in bijection with compositions with a single $k$. 
	Thus, in terms of an $(n+r)$-tiling, the $(k,\dots,k)$ is in bijection with a single red square.
	So, $C_b(n,k,p) = C(1,n-pk, \widehat k)$.
	
	Next, recall that the single red square appearing in $C(1,n-pk,\widehat k)$ can represent $k$, which leads to the interpretation that $C(1,n-pk,\widehat k)$ also counts the compositions of $n-pk + k = n-(p-1)k$ having exactly one $k$.
	By definition, these are enumerated by $E_1(n-(p-1)k,k)$.
\end{proof}

\begin{prop}
	For all $n$ and $k$,
	\[
		C_b(n,k) = C(n,\widehat k) + \sum_{j \geq 0} E_1(n-jk,k).
	\]
\end{prop}

\begin{proof}
	The number of compositions with no $k$s is $C(n,\widehat k)$.
	Additionally, the number of compositions for which the $j$ copies of $k$ are all consecutive is $E_1(n-(j-1)k,k)$.
	Summing these cases and adjusting indices gives the result.
\end{proof}

The following follows directly from inclusion-exclusion, so the proof is omitted.

\begin{prop}
	For all $n,k,p$, we have
	\[
		C_b(n,k,p) = \sum_{j \geq 1} (-1)^{j+1}ja(j,n-k(p+j-1)).
	\]
\end{prop}

Let $C(n,[k])$ denote the number of compositions of $n$ with no parts a multiple of $k$. 
For example, $C(4,[2]) = 3$, counting the compositions $(3,1)$, $(1,3)$, and $(1,1,1,1)$. 

\begin{prop}
	For all $n$ and $k$,
	\[
		C(n,[k]) = F(n+1,k) - F(n+1-k,k).
	\]
\end{prop}

\begin{proof}
	The generating function for the right side of the desired equation is
	\[
		\frac{1-x^k}{1-x-x^2-\cdots-x^k},
	\]
	whereas the generating function for the left hand side is
	\[
		\frac{1}{1-\left(\sum_{i \geq i} x^i\right) + \left(\sum_{j \geq 1} x^{jk}\right)}.
	\]
	After simplifying the latter expression, we obtain the former. 
\end{proof}

We now return to runs of parts. 
Recall that given a composition $\lambda = (\lambda_1,\dots,\lambda_k)$, a \emph{run} in $\lambda$ is a subsequence of the form 
\[
	\lambda_{i-1} \neq \lambda_i = \lambda_{i+1} = \cdots = \lambda_j \neq \lambda_{j+1}
\]
for some $1 \leq i \leq j \leq k$, using the convention $\lambda_0 = \lambda_{k+1} = 0$.
Let $R(n,k)$ denote the total number of runs consisting of part $k$ over all compositions of $n$, without regard to the length of the run.

\begin{prop}
	For all $n$ and $k$, $R(n,k) = a(1,n-k) - a(1,n-2k)$.
\end{prop}

\begin{proof}
	If $k > n-k$, all runs of $k$ have length one.
	The part $k$ can be represented by a red square, combining with white squares representing nonnegative integers less than $k$.
	The numbers of such $(n+r)$-tilings is given by $a(1,n-k)$.
	Also note that if $k > n-k$, then $a(1,n-2k) = 0$.
	
	If $k \leq n-k$, then some of the white tiles have length $k$, thereby being indistinguishable from the red square representing $k$.
	Therefore, some of the $(n+r)$-tilings merely increase the length of existing runs.
	The number of such instances is $a(1,n-2k)$ because increasing the length of a run has the same effect.
	Hence, $a(1,n-k) - a(1,n-2k) = R(n,k)$.
	
\end{proof}

\begin{cor}
	For all $n$ and $k$,
	\[
		R(n,k) = 2^{n-k-2}(n-k+3) - 2^{n-2k-2}(n-2k+3).
	\]
\end{cor}

Extending $R(n,k)$, let $R(n)$ denote the total number of all runs over all compositions of $n$. 

\begin{cor}
	For all $n$,
	\[
		R(n) = \sum_{k \geq 1} a(1,n-(2k-1)).
	\]
\end{cor}

\begin{proof}
	The result follows from observing that
	\[
		R(n) = \sum_{k \geq 1} R(n,k) = \sum_{k \geq 1} a(1,n-k) - a(1,n-2k).
	\]
	In this sum, all terms of the form $a(1,n-2k)$ are cancelled, leaving only those of the form $a(1,n-(2k+1))$, as desired.
\end{proof}

Let $E(n)$ denote the total number of parts used in all compositions of $n$.
This is well-known to be $(n+1)2^{n-2}$, but what is more interesting here is that we can relate $E(n)$ to other aspects of compositions.
For example, note that $C(n) \leq R(n) \leq E(n)$.
In fact, this relationship can be made more explicit.

\begin{lem}
	The number of parts making up the compositions of $n$ is equal to the number of runs of $n$ and of $n-1$.
	In other words, $E(n) = R(n) + R(n-1)$.
\end{lem}

\begin{proof}
	This follows entirely algebraically.
	Observe that
	\[\begin{aligned}
		E(n) &= a_1(1,n-1) \\
			&= \sum_{i \geq 1} a(1,n-i) \\
			&= \left(\sum_{i \geq 1} a(1,n-2i+1)\right) + \left(\sum_{i \geq 1} a(1,n-2i)\right) \\
			&= R(n) + R(n-1),
	\end{aligned}\]
	as desired.
\end{proof}

Next we will refine $R$ by letting $R(n,k,l)$ denote the number of runs of $k$ of length $l$ over all compositions of $n$.

\begin{conj}
	For all $n,k,l$, 
	\[
		R(n,k,l) = a(1,n-kl) - 2a(1,n-(l+1)k) + a(1,n-(l+2)k).
	\]
\end{conj}

\subsection{Pell numbers}

Recall that the sequence of \emph{Pell numbers} is $\{P(n)\}_{n \geq 0}$ where $P(0) = 0$, $P(1) = 1$, and $P(n) = 2P(n-1) + P(n-2)$ for $n \geq 2$.
The Pell numbers have many combinatorial interpretations, such as the sequence of denominators of the continued fraction for $\sqrt{2}$, and the number of $132$-avoiding two-stack-sortable permutations \cite{pell1}.
We prove the following proposition using a generating function argument, but in light of the combinatorial descriptions of both the Pell numbers and the functions $a_s(n,k)$, we invite the reader to provide a combinatorial proof.

\begin{prop}
	For all $n$,
	\[
		P(n) = \sum_{i \geq 0} a_{2i}(2i+1,n-4i)
	\]
\end{prop}

\begin{proof}
	Using the recurrence for $P(n)$, is it straightforward to show that $P(n)$ has generating function
	\[
		\sum_{n \geq 0} P(n)x^n = \frac{1}{1-2x-x^2}.
	\]
	Let
	\[
		u = \frac{1-x}{1-2x}.
	\]
	Then the generating function of the right side of the desired identity is
	\[\begin{aligned}
		\sum_{n \geq 0} u^2 \left(\frac{ux^{2n}}{1-x}\right)^{2n} 
			&= \frac{(1-x)^2}{(1-2x)^2-x^4} \\
			&= \frac{1}{1-2x-x^2},
	\end{aligned}\]
	as desired.
\end{proof}

\subsection{Applications to palindromes}

In \cite{ChinnHeubach}, the authors found a number of generating functions and recurrence relations for palindromic quantities, and provided extensive data.
The $a(r,n)$ functions can be used to provide formulas for this data. 
First, let $m(r,n)$ denote the number of $(n+r)$-tilings when $r$ red squares combine palindromically with the palindromic white tile arrangements.

To illustrate, consider the four palindromes of $6$ without a central part: $33$, $1221$, $2112$, and $111111$.
Combining two red squares, which we will represent by $r$ in the palindromes, with the corresponding tilings of the palindromes, we result in twelve total palindromes:
\[
	\begin{array}{cc}
		3rr3 & r1221r \\
		r33r & r2112r \\
		12rr21 & 111rr111 \\
		1r22r1 & 11r11r11 \\
		21rr12 & 1r1111r1 \\
		2r11r2 & r111111r
	\end{array}
\]

Recall that a white tile of length $k$ can be represented by the positive integer $k$. 
The red square can be represented by any integer whose length is not duplicated by any of the white spaces. 

\begin{prop}
	For all $r$ and $n$,
	\[\begin{aligned}
		m(2r,2n) &= m(2r,2n+1) = a_1(r,\lfloor n/2\rfloor) \\
		m(2r+1,2n) &= a_0(r,n) \\
		m(2r,2n+1) &= 0
	\end{aligned}\]
\end{prop}

\begin{proof}
	We will treat each case separately.
	First let $N = 2n$ and $R = 2r$.
	The palindrome compositions of $N$ are the compositions of $n-k$ paired with their mirror images and central part $2k$ (or no central part if $k=0$), for some $k = 0,\dots,n$. 
	This is counted by 
	\[
		a(r,n) + a(r,n-1) + \cdots + a(r,0) = a_1(r,n).
	\]
	So, $m(R,N) = a_1(r,n)$.
	
	Now suppose $N = 2n$ and $R = 2r+1$. 
	In this case, there is an unpaired red square. 
	It cannot be combined with an existing central part since that would disrupt the palindromicity.
	However, the square can serve as a lone central part. 
	Thus, the number of these is $a_0(r,n)$. 
	Putting this together with the case of $N=2n$ and $R=2r$, we obtain the first two inequalities in the statement of our result.
	
	If $N=2n+1$ and $R=2r$, then all of the palindromic compositions will have a central part. 
	Thus, counting each $a(r,l)$ $l = 0, \dots, n$, we end up with $a_1(r,n)$. 
	
	Finally, suppose $N=2n+1$ and $R=2r+1$.
	In this case, all of the palindromic compositions have a central part.
	Also, there must be unpaired one square.
	However, it is impossible for these to simultaneously exist, so no such compositions exist, completing the proof. 
\end{proof}

\begin{table}
\begin{tabular}{|c|c|c|c|c|c|c|c|c|c|c|} \hline
	\backslashbox{$r$}{$n$} & 0 & 1 & 2 & 3 & 4 & 5 & 6 & 7 & 8 & 9 \\ \hline
	0 & 1 & 1 & 2 & 2 & 4 & 4 & 8 & 8 & 16 & 16 \\ \hline
	1 & 1 & 0 & 1 & 0 & 2 & 0 & 4 & 0 & 8 & 0 \\ \hline
	2 & 1 & 1 & 3 & 3 & 8 & 8 & 20 & 20 & 48 & 48 \\ \hline
	3 & 1 & 0 & 2 & 0 & 5 & 0 & 12 & 0 & 28 & 0 \\ \hline
	4 & 1 & 1 & 4 & 4 & 13 & 13 & 38 & 38 & 104 & 104 \\ \hline
	5 & 1 & 0 & 3 & 0 & 9 & 0 & 25 & 0 & 66 & 0 \\ \hline
	6 & 1 & 1 & 5 & 5 & 19 & 19 & 63 & 63 & 192 & 192 \\ \hline
	7 & 1 & 0 & 4 & 0 & 14 & 0 & 44 & 0 & 129 & 0 \\ \hline
	8 & 1 & 1 & 6 & 6 & 26 & 26 & 96 & 96 & 321 & 321 \\ \hline
\end{tabular}
\caption{$m(r,n)$ for small choices of $r$ and $n$.}\label{tab: m(r,n)}
\end{table}

\begin{rmk}
	We again pause to mention connections with the existing literature.
	The sequence $\{m(3,2n)\}_{n \geq 0}$ appears to coincide with $C_n^{2,2}$ of \cite{heinhuang}, counting certain equivalences classes of objects enumerated by the Catalan numbers.
	Multiple rows of Table~\ref{tab: m(r,n)} suggest that they are the coefficients of $C_2(x)^k$ for various $k$, as described in \cite[Corollary3.16]{heinhuang}.
	We have seen $\{m(4,2n)\}_{n \geq 0}$ before, as $\{a_1(2,n)\}_{n \geq 0}$.
	Again, we invite the reader to establish bijections among the objects counted by these sequences and the Catalan-like objects found in \cite{heinhuang}.
\end{rmk}
\begin{defn}
	For nonnegative integers $n$, let $\Pal(n)$ denote the number of palindromic compositions of $n$ and let $\Pal(n,\widehat k)$ denote the number of palindromes of $n$ with no part $k$.
\end{defn}

It is not difficult to argue that $\Pal(2n) = \Pal(2n+1) = 2^n$, and observe that $a_1(0,n) = 2^n$.

\begin{thm}
	If $n$ and $k$ are nonnegative integers,
	\[
		\Pal(n,k) = \sum_{j \geq 0} (-1)^jm(j,n-2j).
	\]
	In particular, if $n$ and $k$ have the same parity, then
	\[
		\Pal(n,k) = \sum_{j \geq} (-1)^j(a_1(j,n-jk) - a(j,n-(j+1)k)),
	\]
	and if $n$ and $k$ have different parity, then 
	\[
		\Pal(n,k) = \sum_{j \geq 0} (-1)^j a_1(j,n-2k).
	\]
\end{thm}

\begin{proof}
	First consider the case where $n$ and $k$ are both even, say $k = 2j$ for some integer $j$.
	Any palindromic composition of $n$ must then have either no central part (i.e. a central part of $0$) or an even central part.
	That is, a palindromic composition of $n$ is of the form
	\[
		(c_1,c_2,\dots,c_s,\ell, c_s,\dots,c_2,c_1)
	\]
	where $\ell$ is a nonnegative even integer and $(c_1,\dots,c_s)$ is a composition of $n' = \frac{1}{2}(n - \ell)$ with no parts $k$.
	The number of such compositions is $c(n',\widehat{k})$.
	By ranging over all possible $\ell$, and then subtracting the instance where $\ell = k$, we obtain the desired formula. 
	The formula involving $m(\cdot,\cdot)$ follows as a consequence.
	
	The remaining cases (where $n$ and $k$ both odd or have different parity) follow similarly and therefore their details are omitted.
\end{proof}

The above result has the following immediate corollary.

\begin{cor}
	The number of palindromic compositions of $n$ having at least one part $k$ is
	\[
		\Pal(n) - \Pal(n,\widehat k) = \sum_{j \geq 1} (-1)^{j-1}(m(2j-1, n-(2j-1)k) + m(2j,2jk)).
	\]
\end{cor}

\section*{Acknowledgements}

The authors would like to thank George Andrews and Arthur Benjamin for their helpful comments on various parts of initial drafts of this work.
We also thank the anonymous referee for bringing our attention to an assortment of connections within the existing literature.

\bibliographystyle{plain}

\begin{thebibliography}{10}

\bibitem{twotoned}
Arthur~T. Benjamin, Phyllis Chinn, Jacob~N. Scott, and Greg Simay.
\newblock Combinatorics of two-toned tilings.
\newblock {\em Fibonacci Quart.}, 49(4):290--297, 2011.

\bibitem{bjornerwelker}
A.~Bj{\"o}rner and V.~Welker.
\newblock The homology of ``$k$-equal'' manifolds and related partition
  lattices.
\newblock {\em Advances in Mathematics}, 110(2):277--313, 1995.

\bibitem{ChinnHeubach}
Phyllis Chinn and Silvia Heubach.
\newblock Compositions of {$n$} with no occurrence of {$k$}.
\newblock In {\em Proceedings of the {T}hirty-{F}ourth {S}outheastern
  {I}nternational {C}onference on {C}ombinatorics, {G}raph {T}heory and
  {C}omputing}, volume 164, pages 33--51, 2003.

\bibitem{chinnsimay}
Phyllis Chinn and Greg Simay.
\newblock A new family of functions and their relationship to compositions and
  {$k$}-{F}ibonacci numbers.
\newblock In {\em Proceedings of the {F}ortieth {S}outheastern {I}nternational
  {C}onference on {C}ombinatorics, {G}raph {T}heory and {C}omputing}, volume
  196, pages 31--40, 2009.

\bibitem{cuisinaire}
P.Z. Chinn, G.~Colyar, M.~Flashman, and E.~Migliore.
\newblock Cuisenaire rods go to college.
\newblock {\em PRIMUS}, 2:118--130, 1992.

\bibitem{pell1}
Eric~S. Egge and Toufik Mansour.
\newblock 132-avoiding two-stack sortable permutations, {F}ibonacci numbers,
  and {P}ell numbers.
\newblock {\em Discrete Applied Mathematics}, 143(1):72--83, 2004.

\bibitem{heinhuang}
Nickolas Hein and Jia Huang.
\newblock Variations of the {C}atalan numbers from some nonassociative binary
  operations.
\newblock {\em S\'{e}m. Lothar. Combin.}, 80B:Art. 31, 12, 2018.

\bibitem{chebyshev}
Milan Janji\'c.
\newblock An enumerative function.
\newblock 02 2008.
\newblock arXiv:0801.1976.

\bibitem{countingfunction}
Milan Janji\'c and Boris Petkovi\'c.
\newblock A counting function generalizing binomial coefficients and some other
  classes of integers.
\newblock {\em Journal of Integer Sequences [electronic only]}, 17, 02 2014.

\bibitem{ascentsequences}
Sergey Kitaev and Jeffrey Remmel.
\newblock A note on $p$-ascent sequences.
\newblock {\em Journal of Combinatorics}, 8:487--506, 01 2017.

\bibitem{oeis}
N.~J.~A. Sloane.
\newblock The on-line encyclopedia of integer sequences.
\newblock {\em published electronically}, 2019.

\bibitem{generatingfunctionology}
Herbert~S. Wilf.
\newblock {\em Generatingfunctionology}.
\newblock A. K. Peters, Ltd., Natick, MA, USA, 2006.

\bibitem{popularity}
A.F.Y. Zhao.
\newblock Pattern popularity in multiply restricted permutations.
\newblock {\em Journal of Integer Sequences}, 17, 01 2014.

\end{thebibliography}

\end{document}